\documentclass[11pt]{amsart}
\usepackage{amsfonts} 
\usepackage{amsmath}
\usepackage{amssymb}
\usepackage{amsthm}
\usepackage{subfigure}
\usepackage{geometry}
\usepackage{graphicx}
\usepackage{bbm}

\newtheorem{thm}{Theorem}[section]

\newtheorem{lem}[thm]{Lemma}
\newtheorem{defn}[thm]{Definition}

     \newcommand{\norm}[1]{\left\|#1\right\|}
     \newcommand{\abs}[1]{\left|#1\right|}
     
     \newcommand{\floor}[1]{\left \lfloor #1 \right \rfloor}
     \newcommand{\ceil}[1]{\left \lceil #1 \right \rceil}

     \newcommand{\prob}[2]{\mathbb{P}_{#2} \left(#1\right)}
     \newcommand{\gprob}[1]{ {\bf P} \left[ #1 \right]}
     \newcommand{\probGone}[2]{\mathbb{P}_{#2}^{G_1} \left(#1\right)}
     \newcommand{\pr}[2]{{\bf Pr}_{#2}\left(#1\right)}

     \newcommand{\E}[0]{\mathbb{E}}
     \newcommand{\Var}[0]{\text{Var}}
     
     \newcommand{\Z}[0]{\mathbb{Z}}
     \newcommand{\N}[0]{\mathbb{N}}

     \newcommand{\X}[0]{{\bf X}}     
     \newcommand{\vect}[1]{\boldsymbol{#1}}
     \newcommand{\eps}[0]{\epsilon}
     \newcommand{\del}[0]{\delta}
     \newcommand{\gam}[0]{\gamma}
     
     \newcommand{\kmix}[0]{{k_{\text{mix}}} }
     \newcommand{\tmix}[0]{{t_{\text{mix}}} }
     \newcommand{\path}[0]{\mathcal{P}}
     \newcommand{\dist}[0]{\text{dist}}
     
     \newcommand{\tbridge}[0]{T^{(u_b,v_b)}}

     \def\gae{\lower 3pt \hbox{$\ \buildrel {\displaystyle >}\over \sim \ $}} 
     \def\lae{\lower 3pt \hbox{$\ \buildrel {\displaystyle <}\over \sim \ $}} 
          
\begin{document}
  
    \title{\bf Contact Process on a Graph with Communities}
    
\author{David Sivakoff}

    \maketitle
 
 \begin{abstract}
We are interested in the spread of an epidemic between two communities that have higher connectivity within than between them.  We model the two communities as independent Erd\"os-R\'enyi random graphs, each with $n$ vertices and edge probability $p = n^{a-1}$ ($0<a<1$), then add a small set of bridge edges, $B$, between the communities.  We model the epidemic on this network as a contact process (Susceptible-Infected-Susceptible infection) with infection rate $\lambda$ and recovery rate $1$.  If $np\lambda = b > 1$ then the contact process on the Erd\"os-R\'enyi random graph is supercritical, and we show that it survives for exponentially long.  Further, let $\tau$ be the time to infect a positive fraction of vertices in the second community when the infection starts from a single vertex in the first community.  We show that on the event that the contact process survives exponentially long, $\tau\abs{B}/(np)$ converges in distribution to an exponential random variable with a specified rate.  These results generalize to a graph with $N$ communities.
\end{abstract}

\section{Introduction}
Let $G_1 = (V_1, E_1)$ and $G_2 = (V_2,E_2)$ be independent instances of $\mathcal{G}(n,p)$, the Erd\"os-R\'enyi random graph ensemble with $n$ vertices and edge probability $p$.  Construct the graph $G = (V,E)$ such that $V=V_1\cup V_2$ and $E= E_1\cup E_2 \cup B$, where $B \subset V_1 \times V_2$ is a set of `bridge' edges chosen independently of $G_1$ and $G_2$.  When $B$ is a small set of edges relative to $E_i$ ($i=1,2$), then the graph will have two distinct communities with a higher concentration of edges within each community than between the two communities.  During the 2009-2010 Stochastic Analysis program at SAMSI, John McSweeney and Bruce Rogers simulated the contact process (defined carefully below) on this graph as a model for a Susceptible-Infected-Susceptible epidemic in a network with two communities.  Figure~\ref{simulations} depicts their results on a network with $1000$ total vertices ($n=500$ in each community), mean degree $np = 50$ and $\abs{B}=2$ bridge edges.  Each line represents an independent simulation in which initially there are $2$ infected vertices in $V_1$ and all other vertices are healthy; infected vertices become healthy at rate 1 and transmit the infection to their neighbors at rate $\lambda = 0.06$.  In each simulation the infection very quickly reaches a quasi-equilibrium state in the first component $V_1$, then makes a jump to the second component $V_2$ at a random time.  We were motivated by the picture to prove this rigorously, and to determine the limiting distribution of the random jump time.  We also prove that the contact process survives for exponentially long (in $n$) on the random graph before eventually hitting the absorbing state in which all vertices are healthy.

\begin{figure}[ht]
\centering
\includegraphics[width=.8\textwidth]{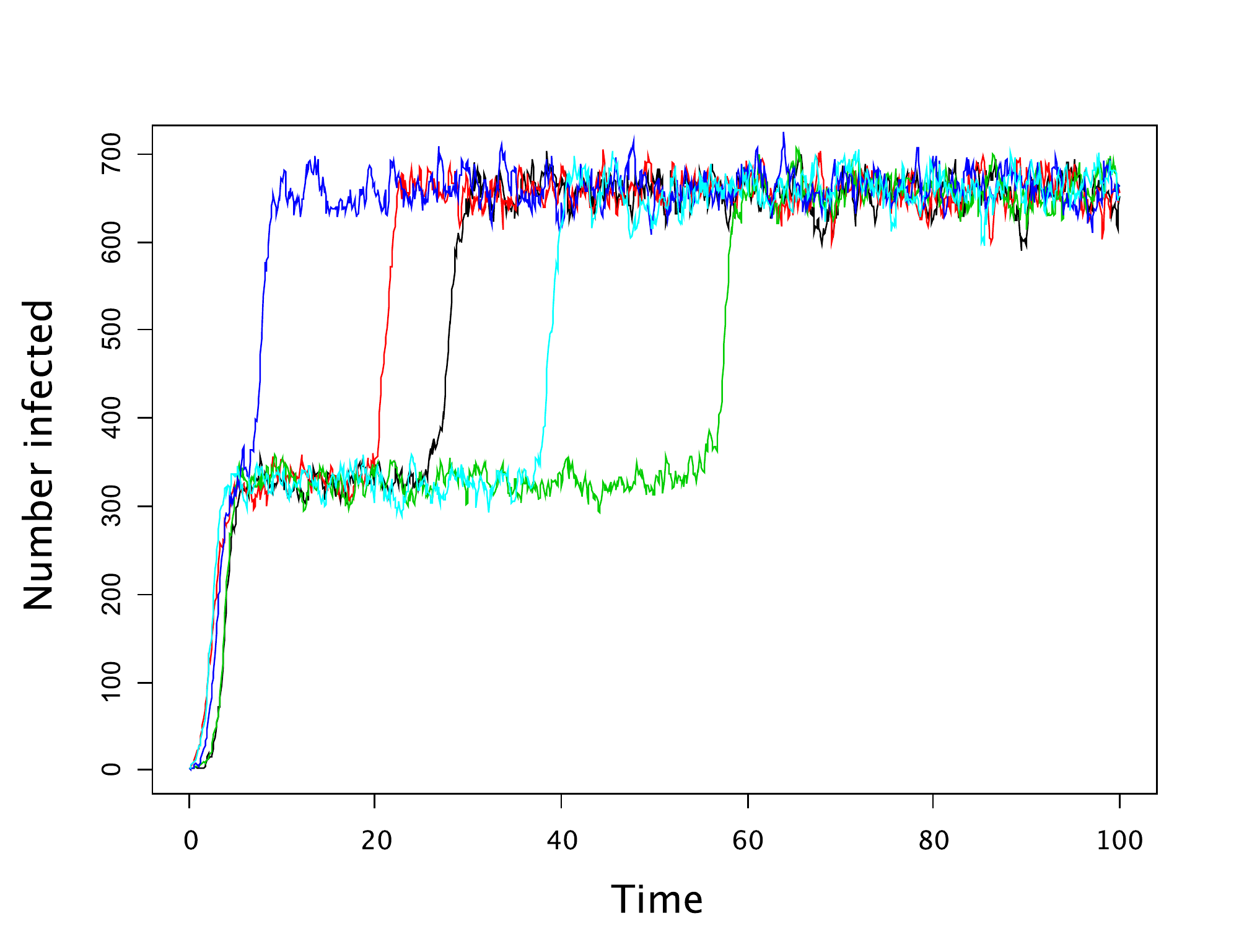}
\caption{Simulation results for the contact process on $G$ with $\abs{B} = 2$, $n = 500$, $p = 0.1$ and $\lambda = 0.06$ so that $b = 3$.  Each line represents an independent trial started from two initially infected vertices in $V_1$. Picture is due to John McSweeney and Bruce Rogers.}
\label{simulations}
\end{figure}

The contact process on a graph $G$ is a continuous time Markov process $\xi_t \subset V$, where $\xi_t$ denotes the set of infected vertices at time $t$.  If $\lambda = \lambda(n,p)>0$ is the infection rate, then an infected vertex sends the infection to each of its neighbors in $G$ according to independent Poisson processes with rate $\lambda$, and it becomes uninfected according to an independent Poisson process with rate $1$.  In effect, a healthy vertex $v\notin\xi_t$ becomes infected at rate $\lambda \abs{\mathcal{N}(v)\cap \xi_t}$, where $\mathcal{N}(v)$ denotes the set of neighbors of $v$ in $G$, while an infected vertex becomes healthy at rate $1$.  This is made rigorous by Harris' graphical construction, which is described in Section~\ref{section:duality}.

The contact process has been recently studied on two different models of power-law random graphs by Chatterjee and Durrett~\cite{CD:2009} and Berger, Borgs, Chayes and Saberi~\cite{BBCS:2005}.  The term power-law random graph refers to the degree distribution having tails that decay like $c k^{-\alpha}$.  For the random graphs considered in~\cite{CD:2009} and~\cite{BBCS:2005} it was shown that the contact process survives on these graphs for a long time for any $\lambda>0$.  This was in contrast to the mean-field calculations of~\cite{PSV1:2001} and~\cite{PSV2:2001}, which predicted that for $\alpha>3$ there is a $\lambda_c>0$ so the contact process will die out quickly for $\lambda < \lambda_c$.  Subsequently, Peterson showed that $\lambda_c>0$ for the contact process on the complete graph with random vertex weights following a power-law~\cite{Peterson:2010}, and in fact explicitly gives the value of $\lambda_c$ in terms of the second moment of the vertex weights.

\subsection{Main Results}
We study the high-degree regime where the mean degree of the random graph scales as $np = n^{a}$ with $a\in (0,1]$.  Since we want the graph to be connected, we need the average degree to be at least $np \geq c \log n$ for $c>1$.  However, part of the proof requires that the random walk on $\mathcal{G}(n,p)$ be very close to uniformly distributed on the vertices after a short amount of time (see Lemmas~\ref{mixing-lem} and~\ref{dual-lem}).  This part of the proof is simplified when $np = n^a$, though a similar method may work for smaller average degrees.  Additionally, in the early stages of the infection process we must have large neighborhoods so that transmission events between neighboring infected vertices are rare.  For smaller degrees, additional care will be needed to guarantee that the infection spreads out quickly.

Before we state our main results, we will need some notation.  Let $\gprob{\cdot}$ denote the law of $G = G_1\cup G_2 \cup B$, where $G_1, G_2 \sim \mathcal{G}(n,p)$ are independent random graphs, and the edges in $B$ are chosen independently of $G_1$ and $G_2$ such that each edge in $B$ has exactly one endpoint in $G_1$ and one in $G_2$. The dependence on $n$ and $p$ has been suppressed from our notation.  Once $G$ is chosen according $\gprob{\cdot}$, it is fixed for all time, so we let $\prob{\cdot}{A}$ be the law of the contact process conditional on $G$ with $\xi_0 = A \subset V$, and abbreviate $\prob{\cdot}{v} := \prob{\cdot}{\{v\}}$.  We will also often want to observe the contact process on only one of the two subcomponents.  That is, we will ignore the edges in $B$, so $\xi_t = \xi_t\cap V_1$ for all $t$ whenever $\xi_0 = A\subset V_1$, and we denote the law of this restricted process conditional on $G$ by $\probGone{\cdot}{A}$.  We will also use $a_n = O(b_n)$ to mean $a_n \leq C b_n$ for $C>0$, $a_n = \Omega(b_n)$ to mean $a_n>c b_n$ for $c>0$, and $a_n = o(b_n)$ to mean $a_n/b_n \to 0$.

%For simplicity we assume that there is just one bridge edge $B = \{(u_b, v_b)\}$ with $u_b \in V_1$ and $v_b \in V_2$.  
First we must guarantee that the contact process on the random graph can survive long enough to spread from the first community to the second.  If $np\lambda = b < 1$, then the contact process is dominated by a subcritical branching process, and dies out quickly, so we assume that $np\lambda = b > 1$. In this case we have the following theorem, which says that the supercritical contact process survives exponentially long on the random graph when it survives for at least $\Omega(\log \log n)$ time.  We assume survival for at least $\Omega(\log \log n)$ time because at this point the infection has either died out (with probability $\approx1/b$) or reached order $\log n$ vertices.
\begin{thm}
\label{survival-thm}
Consider the contact process, $\xi_t$, on the Erd\"os-R\'enyi random graph $G_1 \sim \mathcal{G}(n,p)$ with $np = n^a$, $np\lambda=b>1$, and $\xi_0 = \{v\} \subset V_1$, and let $r= \frac{2}{b-1}\log\log n$.   Then there exist constants $\eta_3, \eps, c>0$ depending on $b$ so that for any $\delta>0$
\begin{equation*}
\gprob{\probGone{\min_{t \in [\eta_3 \log n, e^{cn}]} \abs{\xi_t} \leq \eps n \ \middle | \ \abs{\xi_r}>0}{v} > \delta} \to 0.
\end{equation*}
\end{thm}
The constants $\eta_3$ and $\eps$ in Theorem~\ref{survival-thm} are defined in Lemma~\ref{middle-cp-lem}, and the proof appears thereafter.  We use an approach similar to the one employed by Peterson~\cite{Peterson:2010} to prove exponential survival of the supercritical contact process on the complete graph with random vertex weights.  That is, we show that if the size of the contact process initially exceeds $\gamma n$ for some $\gamma > \eps$, then in a small constant amount of time, the size of the contact process is likely (with probability exponentially close to $1$) to have increased by the end of the time interval without ever having dropped below size $\eps n$.  The main difference between our proof and Peterson's is that we must rely on an isoperimetric inequality (Lemma~\ref{iso-lem}) to guarantee that the contact process has room to expand.

Our main question is, if the infection is started at a single vertex $v_0 \in V_1$, and the infection is able to spread, how long will it be before a positive fraction of the vertices in $V_2$ are infected?  In other words, how long does the infection take to cross a bridge between the two populations and spread throughout the second population?  Theorem~\ref{main-thm} answers this question.

\begin{thm}
\label{main-thm}
Fix $a \in (0,1]$.  Suppose $np = n^a$, $np\lambda = b >1$, and there are $\abs{B} = o(n^a/\log n \log \log n)$ bridges between $V_1$ and $V_2$.  Choose any $v_0 \in V_1$.  There exists $\eps>0$ so that if $\tau := \inf \{t >0 : \ \abs{\xi_t \cap V_2} > \eps n \}$, then for any $x \in [0,\infty)$ and any $\delta>0$
\begin{equation*}
\gprob{\abs{\prob{\frac{\tau}{n^a/\abs{B}} \leq x}{v_0} - \left(1-\frac{1}{b}\right) \left[1 - \exp \left(-\frac{(b-1)^2}{b} x\right)\right]} > \delta} \to 0.
\end{equation*}
\end{thm}
This says that when the infection survives long enough (which happens with probability $\approx 1- 1/b$), the distribution of $\tau$, the time at which the infection has spread to a positive fraction of vertices in $V_2$, is approximately exponential with rate $\abs{B}(b-1)^2/(bn^a)$.  
%Our proof works just the same for as many as $\abs{B} = o(n^a/\log n \log \log n)$ bridges, in which case the distribution of $\tau$ is approximately exponential with rate $\abs{B}(b-1)^2/(bn^a)$ on the event that the infection survives long enough.
The upper bound on the number of bridges that we can accommodate is because our proof requires that the amount of time between the first $\log n$ successive potential transmissions of the infection between the two communities is at least $\Omega(\log \log n)$, and the rate at which such potential transmissions occur is $\abs{B} \lambda = b\abs{B}/n^a$.  The maximum number of bridges allowable to guarantee the separation of timescales seen in Figure~\ref{simulations} should be $O(n^{a})$ as the following mean-field argument demonstrates.  For small times $t$, the number of infected vertices in $V_1$ at time $t$ is approximately $b^t$, since each infected vertex has expected degree $np$ and spreads the infection to each neighbor at rate $\lambda$.  Therefore, the number of infected vertices in $V_1$ will reach $\eps n$ at time $s = \log (\eps n) / \log b$.  The expected rate at which the infection is transmitted from $V_1$ to $V_2$ at time $t$ is approximately $b^t \lambda \abs{B}/ n = b^{t+1} \abs{B}/ n^{1+a}$, so the expected number of times the infection is transmitted from $V_1$ to $V_2$ before time $s$ is
\begin{align*}
\int_0^s \frac{b^{t+1} \abs{B}}{n^{1+a}} dt = \frac{b^{s+1} \abs{B} \log b}{n^{1+a}} \approx \frac{\abs{B} \eps b \log b}{n^a}.
\end{align*}
So, if $\abs{B} \gg n^a$ then the infection is likely to spread to $V_2$ before reaching size $\eps n$ in $V_1$.  It is worthwhile to note that if $V_1$ and $V_2$ are two halves of a homogeneous Erd\"os-R\'enyi random graph with mean degree $n^a$ then $\abs{B} \approx n^{1+a}$, so no separation of timescales should be observable.

Theorem~\ref{main-thm} generalizes easily to $N$ communities, $G_1, \ldots, G_N$, each independent and distributed as $\mathcal{G}(n,p)$ with $np=n^a$ ($N$ is fixed).  If $B_{ij}$ is the (possibly empty) set of bridge edges between communities $i$ and $j$, and $B_i = \cup_j B_{ij}$, then we assume $\max_i \abs{B_i} = o(n^a/\log n \log \log n)$ and for all $1\leq i,j,k,\ell \leq N$ there are constants $c, C>0$ such that whenever $\abs{B_{ij}}> 0$ and $\abs{B_{k\ell}}> 0$ then
\begin{align*}
c \leq \frac{\abs{B_{ij}}}{\abs{B_{k\ell}}} \leq C.
\end{align*}
Under this assumption, all pairs of communities have either a comparable number of bridges, or no bridges (are not directly connected), so we let $\beta_{ij} = \abs{B_{ij}}/ \max_{k,\ell} \abs{B_{k\ell}}$.  We can then define a process $\chi_t \in \{0,1\}^N$ such that $\chi_t(i) = 1$ if and only if community $i$ has at least $\eps n$ infected vertices at time $tn^a/\max_{k,\ell}\abs{B_{k\ell}}$.  If $np\lambda=b>1$ and $\chi_0 = (1, 0, 0, \ldots, 0)$, then $\chi_t$ converges in distribution on a finite time interval to the monotone stochastic process in which $\chi_t(i)$ flips from $0$ to $1$ at rate $\sum_j \chi_t(j) \beta_{ij} (b-1)^2/b$.  That is, at the community level, the infection process resembles an SI epidemic with inhomogeneous infection rates.

The assumption that the number of bridges between connected communities is of a single order of magnitude is not necessary, except that things become a bit more complicated when the number of bridges between communities can span multiple orders of magnitude.  In particular, we need to consider multiple time scales.  The time for the infection to pass from community $i$ to community $j$ is proportional to the `length' of the shortest path between the two, where a path from $i$ to $j$ is a sequence of communities $i = x_0, x_1, \ldots, x_k=j$ such that $\abs{B_{x_\ell x_{\ell+1}}}>0$, the `length' of this path is $1/\min_\ell \abs{B_{x_\ell x_{\ell+1}}}$ and the shortest path is the one with the minimal `length'.  This is because most time will be spent while the infection attempts to cross the narrowest bottleneck.

\subsection{Proof Outline}  For simplicity, we show the proof for the case where there is a single bridge edge, $B = \{(u_b, v_b)\}$ with $u_b \in V_1$ and $v_b \in V_2$.  The generalization to multiple bridges is straightforward, and we comment on this briefly in Section~\ref{main-proof}.

{\bf Step 1.} When the contact process starts from a single vertex, $\xi_0 = \{v_0\} \subset G_1$, for a small amount of time, $r = O(\log \log n)$, $\abs{\xi_t}$ is well approximated by a continuous time branching process with survival probability $1-1/b$, where $b = np\lambda$ (Lemmas~\ref{early-cp-lem} and~\ref{middle-cp-lem}).  By time $r$, the contact process has either died out, or is destined to survive for exponentially long with positive density by Theorem~\ref{survival-thm}.  The main difficulty in this part of the proof is managing `collision' events, where an infected vertex attempts to spread the infection to another already infected vertex.  Once the process grows beyond size $n^a$, these collision events potentially stifle the growth of $\xi_t$.  We overcome this obstacle by the use of the isoperimetric bound given in Lemma~\ref{iso-lem}.  This bound says that as long as the contact process contains at most a small fraction of vertices in $V_1$, then the number of edges between $\xi_t$ and $V\setminus \xi_t$ will be large.  We use this bound to show that while $\abs{\xi_t}$ is not too large, it stochastically dominates a random walk with positive drift.

{\bf Step 2.} We use the self-duality of the contact process, which essentially means that the time reversal of the contact process has the same distribution as the contact process -- we will carefully describe duality in Section~\ref{section:duality}.  The dual process started at time $t>0$, $\{\zeta_s^t\}_{s\in [0,t]}$, is constructed so that if it is started from a single vertex $\zeta_0^t = \{v\} \subset V$, then $\zeta_s^t \cap \xi_{t-s} \neq \emptyset$ for some $s \in (0,t)$ if and only if $v \in \xi_t$.  In the graphical construction of the contact process, edges attempt to transmit the infection at rate $\lambda$ independent of $\xi_t$, so we observe the dual process started from $u_b$ when the bridge edge $(u_b,v_b)$ attempts to transmit the infection.  The dual process will reach size $\Omega(\log n)$ by time $r$ with probability close to $1-1/b$.  The primary difficulty at this point is getting the dual process, which has grown backwards in time, to intersect with the contact process, which has size at least $\eps n$ ($G$ is not the complete graph, so it is possible for all of the vertices in the contact process to be far from the vertices in the dual process).  We do this by coupling the particles in the dual process with a random walk process in which each particle jumps according to an independent simple random walk on the vertices of $G$ and dies at rate $1$.  Those particles which survive for time $\tmix = O(\log \log n)$ and do not collide with other particles will all be well mixed (Lemmas~\ref{mixing-lem} and~\ref{dual-lem}), and each intersects with the contact process with probability larger than $\eps$, which implies that the two processes intersect with high probability.

{\bf Step 3.} At the times when the bridge edge $(u_b, v_b)$ attempts to transmit the infection (a Poisson process $\{T_k^{(u_b,v_b)}\}_{k}$ with rate $\lambda$) we start an independent dual process.  Each dual process to survive to time $r$ will intersect with $\xi_t$ and result in the infection spreading to $v_b$.  In turn, $v_b$ will spread the infection to $V_2$ with probability approaching $1-1/b$ by repeating the first step of the proof.  Therefore, the number of times that the bridge edge must attempt to transmit the infection approaches a Geometric$([1-1/b]^2)$ distribution, and the interarrival times between successive attempts are distributed as independent Exp$(\lambda)$.  So on the event that $\xi_r \neq \emptyset$, the time required to spread the infection to $V_2$ approaches an Exp$(\lambda[1-1/b]^2)$ distribution, and the probability that $\xi_r \neq \emptyset$ approaches $1-1/b$; this is the statement in Theorem~\ref{main-thm}.

Section 2 is devoted to proving two key lemmas about Erd\"os-R\'enyi random graphs with high degree -- an isoperimetric bound and a strong mixing time estimate.  In Section 3 we compare the early stages of the contact process with a branching process and prove Theorem~\ref{survival-thm}.  In Section 4 we prove a key lemma about the dual of the contact process, and in Section 5 we bring everything together for the proof of Theorem~\ref{main-thm}.  The Appendix contains some basic properties of branching processes that are used in Section 3.

{\bf Acknowledgements.} Thank you to Rick Durrett for numerous helpful conversations during the writing of this paper, and for his comments on the final draft.  Thank you to John McSweeney and Bruce Rogers for performing the original simulations of this model.  This work was started at SAMSI during the 2010-2011 Program on Complex Networks.

\section{Isoperimetric Inequality and Mixing Time Bounds}
%%%%%%% Isoperimetric inequality lemma
In this section we prove Lemma~\ref{iso-lem}, which gives a bound on the $\eps$-isoperimetric number (defined below) of an Erd\"os-R\'enyi random graph with large mean degree, and Lemma~\ref{mixing-lem}, which says that the random walk on the high-degree random graph is almost uniformly distributed after a constant number of steps.  These two properties of the random graph, which is fixed for all time, hold with ${\bf P}$-probability tending to $1$ as $n\to\infty$, so when we later consider the contact process on a random graph, we can condition on it having these properties.

\begin{defn}
Define the {\bf $\eps$-isoperimetric number}, $i_{\eps}(G)$, of a graph, $G = (V,E)$ as 
\begin{equation*}
i_{\eps}(G) = \min \left\{ \frac{\abs{\partial U}}{\abs{U}} \ \middle| \ U \subset V, \ \abs{U} \leq \eps \abs{V} \right\}
\end{equation*}
where $\partial U\subset E$ is the set of edges with exactly one end vertex in $U$.
\end{defn}

\begin{lem}
\label{iso-lem}
If $G \sim \mathcal{G}(n,p)$ is an Erd\"os-R\'enyi random graph on $n$ vertices with edge probability $p$ such that $np \geq 28 (\log n)^3$, then for any fixed $\eps > 0$
\begin{equation*}
i_{\eps}(G) \geq (1-\eps) np - (np)^{2/3}
\end{equation*}
with probability $1 - o(1)$ as $n \to \infty$.
\end{lem}
Note that this bound is essentially tight since if $U$ is any deterministic set of $\eps n$ vertices, then $\abs{\partial U} \sim \text{Binomial}(\abs{U} (1-\eps) n,p)$ so $\abs{\partial U} = \abs{U}(1-\eps)np(1+o(1))$ with high probability.  The proof is adapted from the proof of a similar result for the $\frac{1}{2}$-isoperimetric number of random regular graphs in~\cite{bollobas-88}.
\begin{proof}
Denote by $P(u, m)$ the probability of the event that $G$ contains a set of vertices $U \subset V$ with $\abs{U} = u$ and $\abs{\partial U} \leq m$.  We will have proven the lemma once we show that
\begin{equation*}
\sum_{u=1}^{\eps n} P(u, m(u)) = o(1)
\end{equation*}
where $m(u) := u \left[(1-\eps) np - (np)^{2/3} \right]$.  By Markov's inequality
\begin{align*}
P(u, m(u)) &\leq {n \choose u} \sum_{s = 0}^{m(u)} {u(n-u) \choose s} p^s (1-p)^{u(n-u) -s}\\
&= {n \choose u} \prob{X_u \leq m(u)}{}
\end{align*}
where $X_u \sim \text{Binomial}(u(n-u),p)$.  Bernstein's inequality, as it appears in~\cite{bernstein-ineq}, says that
\begin{equation*}
\prob{\abs{X_u - \E X_u} \geq t}{} \leq 2 \exp \left\{ \frac{-t^2}{2 ( \Var X_u + t/3)} \right\}.
\end{equation*}
Let $t = \E X_u - m(u) = u ( \eps n p - up + (np)^{2/3})$.  Note that $t \geq u (np)^{2/3}$, since $u\leq \eps n$, so
\begin{align*}
\prob{X_u \leq m(u)}{} &\leq \prob{\abs{X_u - \E X_u} \geq u (np)^{2/3}}{} \\
&\leq 2 \exp \left\{ \frac{-u^2 (np)^{4/3}}{2 \left[ u(n-u)p(1-p) + u(np)^{2/3}/3 \right]}  \right\} \\
&\leq 2 \exp \left\{ \frac{-u (np)^{2/3}}{2(np)^{1/3} + 2/3}  \right\} \\
&\leq 2 \exp \left\{ - \frac{1}{3} u(np)^{1/3} \right\},
\end{align*}
provided $np > 1/3$.  Applying the assumption that $np \geq 28 (\log n)^3$ completes the proof:
\begin{align*}
\sum_{u=1}^{\eps n} P(u, m(u)) &\leq 2 \sum_{u=1}^{\infty} \left[ n e^{- (np)^{1/3} / 3} \right]^u \\
&= 2 n e^{- (np)^{1/3} / 3} \left(1- n e^{- (np)^{1/3} / 3} \right)^{-1} = o(1).
\end{align*}
\end{proof}

We will make use of a random walk on a random graph in the proof of Theorem~\ref{main-thm}, and we will need a bound on the mixing time of this random walk. Loosely speaking, this is the number of steps that it takes for the random walk to be `close' to its stationary state.  Let
\begin{equation*}
\norm{\mu - \nu}_{TV} := \frac{1}{2} \sum_{v \in V} \abs{\mu(v) - \nu(v)} = \sum_{v \in V} (\mu(v) - \nu(v))_+ = \sup_{A\subset V} \abs{\mu(A) - \nu(A)}
\end{equation*}
denote the total variation distance between two probability measures, $\mu$ and $\nu$, on the vertices of a graph $G = (V,E)$.

Let $X_k$ be a discrete-time, simple random walk on the vertices of $G\sim \mathcal{G}(n,p)$ where $np = n^a$ and $a\in(0,1]$.  We denote the $k$-step transition probabilities of $X_k$ by $P^k(u,v)$ for $u,v \in V$.  Also, we will denote the probability measure on $V$ that corresponds to the $k^{\text{th}}$ step of the random walk started at $u$ by $P^k(u,\cdot)$.  Let $\pi$ be the stationary distribution for this random walk whenever it is uniquely defined.  Note that the random walk is ergodic (so $\pi$ is unique) asymptotically almost surely because $X_k$ is aperiodic ($G$ is not bipartite) and irreducible ($G$ is connected) with probability superpolynomially close to $1$.

%Then the mixing time of $X_k$, for our purposes, is defined as follows.
%\begin{defn}
%The {\bf mixing time}, {\em $\kmix (\alpha)$}, of the random walk, $X_k$, on the random graph $G \sim \mathcal{G}(n,p)$ with $np = n^{a}$, $a \in (0,1]$ is
%\begin{equation*}
%k_{\text{\em mix}}(\alpha) = \inf \left\{k \ \middle| \ \sup_{u \in V} \norm{P^k(u,\cdot) - \pi}_{TV} \leq \frac{1}{n^{1+\alpha}} \right\}
%\end{equation*}
%where $\alpha >0$, and $\pi$ is the unique stationary distribution of $X_k$.
%\end{defn}

Typically, the mixing time for a random walk, $X_k$, is defined to be the smallest $k$ such that $\sup_{u\in V} \norm{P^k(u, \cdot) - \pi} \leq \alpha$ for some fixed $\alpha \in (0,1/2)$.  In the proof of Lemma~\ref{dual-lem} we need tight control on $P^k(u,\cdot)$ for many vertices, $u$, simultaneously, so our goal is to show that for some constant $\kmix$ depending only on $a$,
\begin{equation*}
\sup_{u\in V}\abs{P^{\kmix}(u, A) - \frac{\abs{A}}{n}} \leq \abs{A} o(n^{-1})
\end{equation*}
for every $A\subset V$.  To say this we need the following lemma, which says that the stationary distribution of $X_k$ is almost uniform.
\begin{lem}
\label{stationary-lem}
Let $G\sim \mathcal{G}(n,p)$ with $np = n^a$, $a \in (0,1]$, and let $\pi$ be the stationary distribution for the simple random walk on $G$ (conditional on its existence and uniqueness).  If $\mu$ is the uniform probability distribution on $G$ ($\mu(v) = n^{-1}$ for all $v\in V$), then
\begin{equation*}
\abs{\pi(A) - \mu(A)} \leq 3n^{-(1+a/3)} \abs{A}
\end{equation*}
for all $A\subset V$ with probability $1 - \exp [-\Omega(n^{a/3}) ]$.
\end{lem}
\begin{proof}
It is easily verified that $\pi(v) = \deg(v) / (2\abs{E})$ for all $v\in V$.  By Bernstein's inequality~\cite{bernstein-ineq}, $\deg(v) \in [n^a - n^{2a/3}, n^a + n^{2a/3}]$ for all $v\in V$ with probability $1 - \exp [-\Omega(n^{a/3}) ]$.  This implies that, for all $A\subset V$ and all sufficiently large $n$ ($\geq 2^{3/a}$),
\begin{align*}
\pi(A) &\leq \frac{\abs{A}}{n} \left[\frac{1 + n^{-a/3}}{1 - n^{-a/3}} \right] \\
& \leq \frac{\abs{A}}{n} \left[ 1 + 3 n^{-a/3} \right]
\end{align*}
with probability $1 - \exp [-\Omega(n^{a/3}) ]$.  Likewise, we have the corresponding lower bound $\pi(A) \geq (\abs{A}/n)[1 - 3 n^{-a/3}]$ for all $A$ with high probability, which proves the lemma, since $\mu(A) = \abs{A}/n$.
\end{proof}
Note that Lemma~\ref{stationary-lem} is slightly stronger than the statement that
\begin{equation*}
\norm{\pi - \mu}_{TV} \leq 3 n^{-a/3}
\end{equation*}
with probability $1 - \exp [-\Omega(n^{a/3}) ]$.

In~\cite{Dou:1992} it is proved that when $G\sim \mathcal{G}(n, d/(n-1))$, $d \in \{1, \ldots, n-1\}$, and $k \geq 0$ and $v\in V$ are fixed, then
\begin{equation*}
\E \norm{P^k(v,\cdot) - \mu}_{TV} \leq c \left( \frac{n}{d^k} + \frac{1}{d}\right)^{1/2}
\end{equation*}
for some absolute constant $c$.  (Note that $\mu$, the uniform probability distribution, appears in the statement, and not $\pi$.)  The problem is that this result does not provide a bound on the mixing time, which entails taking the supremum over all initial locations for the random walk.  This presents a problem for us, since we want to say that many independent random walks, started from different locations, will all be well mixed at the same time.  To remedy this for high degree random graphs, we have the next lemma.
\begin{lem}
\label{rw-tvd-lem}
Let $G\sim \mathcal{G}(n,p)$ with $np = n^{a}$ for $a\in(0,1]$.  If $\kappa = \floor{1+1/a}$, then
\begin{equation*}
\sup_{v\in V} \norm{P^{\kappa}(v,\cdot) - \pi}_{TV} = O(n^{-a/3})
\end{equation*}
with probability $1-o(1)$.
\end{lem}
The definition of $\kappa$ in Lemma~\ref{rw-tvd-lem} is such that $\kappa$ is the smallest integer strictly larger than $1/a$.
\begin{proof}
We will employ a simple path counting argument similar to an argument used by Lubetzky and Sly~\cite{LS:2010} to prove cutoff for the random walk on random regular graphs.  Let $\path_k(u,v)$ denote the number of paths in $G$ of length $k$ that start at $u$ and end at $v$.  As in the proof of Lemma~\ref{stationary-lem}, we can use Bernstein's inequality to bound the maximum and minimum degrees of $G$ as  $\deg(v) \in [n^a - n^{2a/3}, n^a + n^{2a/3}]$ for all $v\in V$ with probability $1 - \exp [-\Omega(n^{a/3}) ]$.  Observe that the probability that the random walk, $X_k$, traverses any path of length $k$, $(v_0,v_1),(v_1,v_2),\ldots,(v_{k-1},v_k) \in E$, is
\begin{equation*}
P^1(v_0,v_1) \cdot P^1(v_1,v_2) \cdots P^1(v_{k-1},v_k) \geq \frac{1}{n^{ka}} (1 - O(kn^{-a/3}))
\end{equation*}
with probability $1 - \exp [-\Omega(n^{a/3}) ]$.  Therefore,
\begin{equation*}
P^k(u,v) \geq \path_k(u,v)\frac{1}{n^{ka}} (1 - O(kn^{-a/3}))
\end{equation*}
for any $u,v\in V$ with high probability.  To obtain a lower bound on $\path_k(u,v)$, we introduce the following notation for balls of radius $k$ in $G$.  For any $u\in V$ and $k\geq 0$, let
\begin{align*}
B_k(u) &:= \{ v\in V \ | \ \dist(u,v)\leq k \} \\
\partial B_k(u) &:= B_k(u) \setminus B_{k-1}(u)
\end{align*}
where $\dist(u,v)$ denotes the length of the shortest path in $G$ from $u$ to $v$, and $\partial B_0(u) := \{u\}$.

For a pair of vertices $u, v \in V$, let $k_u = \ceil{\frac{\kappa-1}{2}}$ and $k_v = \floor{\frac{\kappa - 1}{2}}$, so that $k_u + k_v = \kappa -1$.  We will construct a ball of radius $k_u$ around $u$, then remove this ball from $V$ and construct a second ball around $v$ of radius $k_v$ whenever $v \notin B_{k_u}(u)$.  Since most vertices in $V$ are not within $k_u$ of $u$, we will have two disjoint sets of vertices for most pairs $u,v$.  The number of paths between $u$ and $v$ in $G$ is at least the number of edges between the vertex boundaries of these two balls, so we seek a uniform lower bound on this quantity.

To obtain bounds on $\abs{\partial B_{k_u}(u)}$, we start at $u$ and reveal edges layer by layer.  To start, by Bernstein's inequality,
\begin{align*}
n^a - n^{2a/3} \leq \abs{\partial B_{1}(u)} \leq n^a + n^{2a/3}
\end{align*}
with probability $1 - \exp [-\Omega(n^{a/3}) ]$.  By induction on $k$, we wish to show that $\abs{\abs{\partial B_k(u)} - n^{ak}} = O(n^{ak-a/3})$ for all $k\leq k_u$.  Assume that
\begin{align*}
\abs{\partial B_{k-1}(u)} &\geq n^{(k-1)a}(1 - O(n^{-a/3}))
\\ &\text{ and} \\
\abs{B_{k-1}(u)} &\leq  n^{(k-1)a}(1 + O(n^{-a/3})).
\end{align*}
Since $k\leq k_u$, we have that $(k-1)a<1/2$, so $\abs{V \setminus B_{k-1}(u)} = n - O(n^{1/2})$.  This means that there are at least $n^{(k-1)a+1}(1 - O(n^{-a/3}))$ and at most $ n^{(k-1)a+1}(1 + O(n^{-a/3}))$ potential edges between $\partial B_{k-1}(u)$ and $V\setminus B_{k-1}(u)$.  Therefore, by Bernstein's inequality,
\begin{align*}
\abs{\partial B_{k}(u)} &\geq n^{(k-1)a+1}p (1 - O(n^{-a/3})) - \left[n^{(k-1)a+1}p\right]^{2/3} = n^{ka}(1 - O(n^{-a/3}) \\
\end{align*}
with probability $1 - \exp[ -\Omega(n^{ka/3})]$, and likewise
\begin{align*}
\abs{\partial B_{k}(u)} &\leq n^{ka}(1 + O(n^{-a/3}))
\end{align*}
with probability $1 - \exp[ -\Omega(n^{ka/3})]$.  It immediate follows that $\abs{B_k(u)} \leq n^{ka}(1+O(n^{-a/3}))$, which concludes the induction argument.  Therefore, $\abs{\abs{\partial B_{k_u}(u)} - n^{k_u a}} = O(n^{k_u a - a/3})$ with probability $1 - \exp[-\Omega(n^{a/3})]$.

After exposing ${B_{k_u}(u)}$ we can employ the same argument starting from $v\in V\setminus B_{k_u}(u)$, but at each step avoiding the vertices in $B_{k_u}(u)$.  Let
\begin{align*}
B_{k}^{u}(v) &:= \{w \in V\setminus B_{k_u}(u) \ | \ \dist_{u}(v,w) \leq k\} \\
\partial B_k^u(v) &:= B_k^u(v)\setminus B_{k-1}^u(v),
\end{align*}
where $\dist_u(v,w)$ denotes the graph distance on the maximal subgraph of $G$ with vertex set $V\setminus B_{k_u}(u)$.  Note that when $a > 1/2$, $k_v = 0$, so $B_{k_v}^u(v) = \{v\}$. When $a \leq 1/2$, removing the vertices in $B_{k_u}(u)$ from $V$ does not affect any of the estimates made above, because $\abs{B_{k_u}(u)} = O(n^{k_u a}) = O(n^{(1+a)/2}) = O(n^{1 - a/3})$, so that $\abs{V \setminus (B_{k_u}(u) \cup B^u_{k-1}(v))} = n(1 - O(n^{-a/3}))$ at the induction step.  Therefore, $\abs{\abs{\partial B^u_{k_v}(v)} - n^{k_v a}} = O(n^{k_v a - a/3})$ with probability $1 - \exp[-\Omega(n^{a/3})]$.

Now we observe that every edge between the sets $\partial B_{k_u}(u)$ and $\partial B^u_{k_v}(v)$ contributes at least one path of length $\kappa$ to $\path_{\kappa}(u,v)$.  There are at least $\abs{\partial B_{k_u}(u)} \cdot \abs{\partial B^u_{k_v}(v)} \geq n^{(\kappa - 1)a}(1 - O(n^{-a/3}))$ potential edges between the two sets.  So by again applying Bernstein's inequality we have that
\begin{equation*}
\path_{\kappa}(u,v) \geq n^{\kappa a - 1}(1-O(n^{-a/3})) = \omega(1)
\end{equation*}
with probability $1 - \exp[-\Omega(n^{a/3})]$, and therefore that
\begin{equation*}
P^{\kappa}(u,v) \geq \frac{1}{n}(1 - O(n^{-a/3}))
\end{equation*}
with probability $1 - \exp[-\Omega(n^{a/3})]$.  By applying a union bound over all pairs $u\in V$ and $v\in V\setminus B_{k_u}(u)$, this bound on $P^{\kappa}$ holds for all such $u,v$ with probability $1 - o(1)$.  Therefore
\begin{align*}
\sup_{u\in V} \norm{\mu - P^{\kappa}(u, \cdot)}_{TV} &= \sup_{u\in V} \sum_{v\in V} \left(\frac{1}{n} - P^{\kappa}(u,v) \right)_+\\
&\leq \sup_{u\in V} \sum_{v\in V\setminus B_{k_u}(u)} \left(\frac{1}{n} - \frac{1}{n}(1 - O(n^{-a/3}))\right)_+ + \frac{1}{n} \abs{B_{k_u}(u)} \\
& = O(n^{-a/3})
\end{align*}
with probability $1-o(1)$.  Applying Lemma~\ref{stationary-lem} and the triangle inequality completes the proof.
\end{proof}

An immediate consequence of Lemmas~\ref{stationary-lem} and~\ref{rw-tvd-lem} is the following lemma, which we will use in the proof of Theorem~\ref{main-thm}.

\begin{lem}
\label{mixing-lem}
Let $G\sim \mathcal{G}(n,p)$ with $np = n^a$ for $a \in (0,1]$, and $P^k(u,\cdot)$ be probability measure on $V$ corresponding to the $k^{\text{th}}$ step of the simple random walk on $G$ started at $u$.  Let $\pi$ be the stationary distribution of the random walk, and $\mu$ be the uniform probability measure on $V$.  If $\kmix = 12\ceil{1/a}^2$ then
\begin{equation*}
\sup_{u\in V} \norm{P^{\kmix}(u,\cdot) - \pi}_{TV}  = O(n^{-2})
\end{equation*}
with probability $1-o(1)$.  Furthermore, 
\begin{equation*}
\sup_{u\in V} \abs{P^{\kmix}(u,A) - \mu(A)} \leq 4n^{-(1+a/3)} \abs{A}
\end{equation*}
for all $A\subset V$ with probability $1-o(1)$.
\end{lem}
\begin{proof}
The first equation is due to the following fact, which is a consequence of Lemma 4.11 and equation (4.31) in~\cite{mixing:2009}.  If $\sup_{u\in V} \norm{P^k(u,\cdot) - \pi} \leq \alpha$ for some $k\in \mathbb{N}$ and $\alpha \in (0, 1/2)$, then for any integer $\ell > 0$, 
\begin{equation*}
\sup_{u \in V} \norm{P^{\ell k}(u,\cdot) - \pi} \leq (2 \alpha)^{\ell}.
\end{equation*}
In our case, by Lemma~\ref{rw-tvd-lem}, $\alpha = \text{Const}\cdot n^{-a/3}$, $k = \kappa \leq 2 \ceil{1/a}$ and $\ell = 6 \ceil{1/a}$.  The second equation follows from the first and Lemma~\ref{stationary-lem}.
\end{proof}

%%%%%%%  Section: Survival
\section{Survival of the Infection}
	In this section we will show that the contact process on an Erd\"os-R\'enyi random graph survives for exponentially long whenever it does not die out very quickly.  We will do this by making use of three comparisons with branching processes.  We will use lower and upper bounding branching processes to carefully control the behavior of the contact process started from a single vertex in its early stages.  Then we will use a second lower bounding branching process to show that the contact process will survive to linear size whenever it survives beyond its initial growth stage.  The relevant facts about these branching processes can be found in the Appendix.  Finally, once the contact process occupies a positive fraction of the vertices, we will compare it to a random walk with positive drift to show that it will continue to occupy a positive fraction of vertices for exponentially long.

% Lemma for early stage of CP
\begin{lem}
\label{early-cp-lem}
Let $r = \frac{2}{b - 1 - 2(np)^{-1/3}} \log\log n = \frac{2}{b-1}\log\log n + o(1)$, $\eta_1 = \frac{6b}{(b-1)\log b}$ and $\eta_2 = 3/\log b$.  If the contact process starts with a single vertex, $\xi_0 = \{v\}$, then
\begin{align}
\label{cp-upper-bd}
\prob{\sup_{t\leq r} \abs{\xi_t} > 2\eta_1 (\log n)^4}{v} = O(n^{-2}),\\
\prob{0<\abs{\xi_r} \leq \eta_2 \log n}{v} = O((\log n)^{-1}), \\
\prob{\abs{\xi_r} = 0}{v} = \frac{1}{b} + O((\log n)^{-1}).
\end{align}
with ${\bf P}$-probability $1-o(1)$.
\end{lem}

%proof of lemma for early stage.
\begin{proof} By Bernstein's inequality~\cite{bernstein-ineq}, the degree of any vertex in $G_1$ or $G_2$ is in the interval $[np - (np)^{2/3}, np + (np)^{2/3}]$ with {\bf P}-probability $1 - \exp[-\Omega((np)^{1/3})]$.  For each $n$, let $Y_t^n$ be a branching process in which each individual gives birth to a single offspring at rate $\beta^n_{Y} = b - 2(np)^{-1/3}$ and each individual dies at rate $1$.  Likewise, let $Z_t^n$ be a branching process in which each individual gives birth to a single offspring at rate $\beta^n_Z = b + (np)^{-1/3}$ and each individual dies at rate $1$.  It is clear that $Z_t^n$ stochastically dominates $\abs{\xi_t}$ restricted to the graph $G_1$ (ignoring edges between $G_1$ and $G_2$) for all time provided $Z_0^n = \abs{\xi_0}$.  By Lemma~\ref{bp-max-size}, when $r_1 = \frac{2}{\beta^n_Z - 1} \log \log n$ and $\eta_1 = \frac{6b}{(b-1)\log b} > \frac{6\beta^n_Z}{(\beta^n_Z-1)\log \beta^n_Z}$,
\begin{align*}
\prob{\sup_{t\leq r_1} \abs{\xi_t} > \eta_1 (\log n)^4}{v} &\leq \prob{\sup_{t\leq r_1} Z_t^n > \eta_1 (\log n)^4}{1}  = O(n^{-2}).
\end{align*}
We want to extend this bound slightly to time $r = \frac{2}{\beta^n_Y - 1} \log \log n$.  To do so, we observe that $r = r_1 + 3 (np)^{-1/3} \log \log n + o((np)^{-1/3})$.  The number of birth events during a time interval of length $r - r_1$ for the process $Z_t^n$, when $Z_t^n \leq 2 \eta_1 (\log n)^4$ throughout the interval, is stochastically bounded above by a Poisson($4 b \eta_1 (\log n)^4 (r - r_1)$) random variable for large $n$ (the leading coefficient of $4$, rather than $2$, over-compensates for the difference in rates between $\beta^n_Z$ and $b$).  Therefore, the probability of there being more than $7$ births during this interval is at most $O(n^{-2})$.  If at the start of the time interval $Z_t^n \leq \eta_1 (\log n)^4$, then the branching process cannot have exceeded size $2 \eta_1 (\log n)^4$ at any time during the interval, so we have equation~(\ref{cp-upper-bd}).

This means that with high probability, at any time $t \in [0,r]$ and for any vertex $v \in V_1$, we have that $\abs{\mathcal{N}(v) \cap \xi_t} \leq 2\eta_1 (\log n)^4 < (np)^{2/3}$ for sufficiently large $n$.  Therefore, for all $t\in [0,r]$, $\abs{\mathcal{N}(v) \setminus \xi_t} > np - 2(np)^{2/3}$, and $\abs{\xi_t}$ stochastically dominates~$Y_t^n$.  If $\eta_2 = 3/\log b < 3/\log \beta_Y^n$, then by Lemma~\ref{bp-size-lem} and equation~(\ref{cp-upper-bd})
\begin{align*}
\prob{\abs{\xi_r} \leq \eta_2 \log n}{v} &\leq \prob{Y_r^n \leq \eta_2 \log n}{1} + O(n^{-2}) \\
& = \prob{Y_r^n = 0}{1} + O((\log n)^{-1}) \\
& \leq \frac{1}{b - 2(np)^{-2/3}} + O((\log n)^{-1}) \\
& = \frac{1}{b} + O((\log n)^{-1}).
\end{align*}
The third line above follows because the extinction probability for $Y_t^n$ is $1/\beta_Y^n$, and the last line follows by expanding the geometric series for the first term.  Likewise, we have
\begin{align*}
\prob{\abs{\xi_r} \leq \eta_2 \log n}{v} &\geq \prob{Z_r^n = 0}{1} + \prob{0<\abs{\xi_r} \leq \eta_2 \log n}{v} \\
&\geq \prob{Z_{r_1}^n = 0}{1} + \prob{0<\abs{\xi_r} \leq \eta_2 \log n}{v} \\
& = \frac{1 - e^{-(\beta_Z^n -1) r_1}}{\beta_Z^n - e^{-(\beta_Z^n -1) r_1}} + \prob{0<\abs{\xi_r} \leq \eta_2 \log n}{v} \\
& = \frac{1}{b} - O((\log n)^{-2}) + \prob{0<\abs{\xi_r} \leq \eta_2 \log n}{v}.
\end{align*}
The second line follows because the events $\{Z_t = 0\}$ are increasing in $t$, and the third line is an exact computation of the transition probability.  Combining the last two inequalities (and equation~(\ref{cp-upper-bd}) to guarantee that $\xi_t$ dominates $Y_t^n$) completes the proof. 
\end{proof}

By Lemma~\ref{iso-lem}, the $\eps$-isoperimetric number of $G_j$ for $j=1,2$ is bounded below as
\begin{equation*}
i_{2\eps}(G_j) \geq (1-2\eps)np - (np)^{2/3}
\end{equation*}
for asymptotically almost every $G_j$.  We choose $\eps = \frac{1}{4}(1-1/b^{1/3})>0$ so that \begin{align}
\label{eps-def}
(1-2\eps - (np)^{-1/3} - O((\log n)^4(np)^{-1}))b > (1-3\eps)b > 1
\end{align}
for sufficiently large $n$.  Then Lemma~\ref{iso-lem} applied to $i_{2\eps}(G_j)$ implies that \begin{align}
\label{edge-boundary}
\abs{\mathcal{N}(\xi_t) \setminus \xi_t} \geq (1-3\eps) np \abs{\xi_t}
\end{align}
for large $n$ whenever $\abs{\xi_t} \leq 2\eps n$ for asymptotically almost every $G_j$.  The term $O((\log n)^4(np)^{-1})$ appears in~(\ref{eps-def}) because in the proof of Theorem~\ref{main-thm} it will be necessary to avoid as many as $O((\log n)^4)$ vertices that have been observed by the dual process defined in the next section.  Since~(\ref{edge-boundary}) is a property of the graph, which is fixed for all time, we will assume that this inequality holds for the remainder of the proof.  This means that during the time interval $[r,T_{2\eps}]$, where
\begin{align*}
T_{2\eps} = \inf\{t>r : \abs{\xi_t} > 2\eps n\},
\end{align*}
$\abs{\xi_t}$ stochastically dominates a branching process with per-capita birth rate $(1-3\eps) b>1$ and death rate $1$.  Denote this branching process by $W_t$.  Then
\begin{align*}
\prob{T_{2\eps} \geq s \ \middle|\ \eta_2 \log n < \abs{\xi_r} < \eta_1 (\log n)^4}{} &\leq \prob{W_s \leq 2 \eps n \ \middle| \ W_r = \eta_2 \log n}{} \\
&\leq \prob{W_s \leq 2 \eps n \ \middle| \ W_r = 1}{}^{\eta_2 \log n},
\end{align*}
because the event that $W_s \leq 2\eps n$ given that $W_r = \eta_2 \log n$ implies that all of the $\eta_2 \log n$ families at time $r$ must not have exceeded size $2 \eps n$ by time $s$.  If we let $s = \frac{3}{(1-3\eps)b-1} \log n + r$, then Lemma~\ref{bp-size-lem} implies that
\begin{align*}
\prob{W_s \leq 2 \eps n \ \middle| \ W_r = 1}{}^{\eta_2 \log n} &\leq \left(\frac{1}{(1-3\eps)b} + O(n^{-2})\right)^{\eta_2 \log n} \\
&\leq n^{- 3 - 3\log_b (1-3\eps)} \exp [{O(n^{-2} \log n)}] \\
& = O(n^{-2}).
\end{align*}
In the second line above we used the bound $1+x \leq e^x$, and in the last line we used the fact that $\eps = \frac{1}{4}(1- {1/b^{1/3}})$ so $-3\log_b(1-3\eps) < 1$.  Together with Lemma~\ref{early-cp-lem}, the last two inequalities imply the next lemma.
\begin{lem}
\label{middle-cp-lem}
With $r$ defined as in Lemma~\ref{early-cp-lem}, $\eps = \frac{1}{4}(1- {1/b^{1/3}})$ and $\eta_3 = \frac{4}{(1-4\eps)b -1}$ then
\begin{align}
\prob{T_{2\eps} \geq \eta_3 \log n \ \middle | \ \abs{\xi_r}>0}{v} = O((\log n)^{-1}).
\end{align}
\end{lem}

This means that if the contact process is able to survive to time $r$, which happens with probability $1/b + O((\log n)^{-1})$, then it will reach size $2\eps n$ before time $\eta_3 \log n$.  Theorem~\ref{survival-thm} shows that if the contact process, $\xi_t$, survives to time $r$, then it will continue to survive for exponentially long.  In particular, it will survive long enough to spread to the second component.

%%% Contact Process Survival Prop
\vspace{.2cm}
\noindent
{\bf Theorem~\ref{survival-thm}. }{\it
Consider the contact process, $\xi_t$, on the Erd\"os-R\'enyi random graph $G_1\sim \mathcal{G}(n,p)$ with $np=n^a$, $np\lambda = b$ and $\xi_0 = \{v\}\subset V_1$, and let $r=\frac{2}{b-1} \log \log n$.  Then for $\eta_3>0$ and $\eps>0$ as in Lemma~\ref{middle-cp-lem} and a constant $c>0$ depending on $b$, for any $\delta > 0$
\begin{equation*}
\gprob{\probGone{\min_{t \in [\eta_3 \log n, e^{cn}]} \abs{\xi_t} \leq \eps n \ \middle | \ \abs{\xi_r}>0}{v} > \delta} \to 0.
\end{equation*}
}

For the proof of Theorem~\ref{main-thm}, we will apply Theorem~\ref{survival-thm}, but with as many as $O((\log n)^4)$ vertices removed from $G_1$ at any time during the process.  This added restriction has no effect on the proof of Theorem~\ref{survival-thm}.  Our approach is similar to that of~\cite{Peterson:2010} for the contact process on the complete graph with random edge weights.

\begin{proof}
We intend to show that there exist constants $\del \in (\eps, 2\eps]$ and $\tau, C>0$ so that 
\begin{equation}
\label{survival-intervals}
\inf_{x>\del n} \prob{\abs{\xi_{\tau}} > \del n, \ \min_{t \in [0,\tau]} \abs{\xi_t} > \eps n \ \middle | \ \abs{\xi_0} = x}{} \geq 1 - e^{-Cn},
\end{equation}
for all sufficiently large $n$.  This means that with probability exponentially close to $1$, if the size of the contact process initially exceeds $\del n$, then at time $\tau$ the size of the contact process will again exceed $\del n$ and will not have dropped below $\eps n$ along the way.  By subdividing the time interval $[0,e^{Cn/2}]$ into $e^{Cn/2}/\tau$ intervals of length $\tau$, this implies that
\begin{equation}
\label{long-time-survival}
\sup_{x>\del n} \prob{\min_{t \in [0, e^{Cn/2}]} \abs{\xi_t} \leq \eps n \ \middle| \  \abs{\xi_0} = x}{} \leq \frac{1}{\tau}e^{-Cn/2}.
\end{equation}
By Lemma~\ref{middle-cp-lem}, $T_{2\eps} < \eta_3 \log n$ with probability $1-o(1)$ conditional on $\abs{\xi_r}>0$.  The Strong Markov Property and equation~(\ref{long-time-survival}) imply the result with $c = C/2$.

To prove equation (\ref{survival-intervals}), first we observe that by monotonicity of the contact process, it suffices to prove the statement with initial size $x = \del n$.  We will actually prove the stronger statement that for $\gam = \min \{ \eps (1-4\eps) b, \ 2\eps\}$ (recall that $(1-4\eps)b = b^{2/3}>1$), there exist $\tau, C>0$  and $\del \in (\eps, \gam)$ such that
\begin{equation}
\label{survival-short-intervals}
\prob{\abs{\xi_\tau} > \del n, \ \abs{\xi_t} \in [\eps n, \gam n]  \ \forall t \in [0,\tau] \ \middle| \ \abs{\xi_0} = \del n}{} \geq 1 - e^{-Cn}.
\end{equation}
The difference between the events in equations (\ref{survival-intervals}) and (\ref{survival-short-intervals}) is that in the latter the size of the contact process is also not allowed to exceed size $\gam n$.

The total jump rate of $\abs{\xi_t}$ is at most $O(n)$ for all $t$.  So, for sufficiently small $\tau>0$ (depending on $b$), the probability that the number of jumps in the time interval $[0,\tau]$ exceeds $\frac{1}{4}(\gam - \eps) n$ is at most $e^{-C' n}$ for some $C' >0$ depending on $\tau$ and $b$.  If we choose $\del = \frac{1}{2}(\eps + \gam)$, then $\abs{\xi_0} = \del n$ implies that $\abs{\xi_t} \in [\eps n, \gam n]$ for all $t \leq \tau$ with probability exceeding $1 - e^{-C'n}$.  While $\abs{\xi_t} \in [\eps n, \gam n]$, the maximum rate at which $\abs{\xi_t} \to \abs{\xi_t}-1$ is $\gam n\leq \eps (1-4\eps)b n < \eps (1-~3\eps)b n$, which is the minimum rate at which $\abs{\xi_t} \to \abs{\xi_t} + 1$ in this interval.  Therefore, given that $\abs{\xi_t} \in [\eps n, \gam n]$ for all $t \leq \tau$, $\abs{\xi_t}$ stochastically dominates a random walk with positive drift up to time $\tau$.  Specifically, $\abs{\xi_t}$ stochastically dominates $X_t$, where $X_t$ is the continuous time random walk that jumps to $X_t +1$ at rate $\eps (1-3\eps)bn$ and to $X_t - 1$ at rate $\eps(1-4\eps)b n$.  By a standard large deviations argument for random walks, $\prob{X_\tau \leq \del n}{\del n} \leq e^{-C'' n}$, where $C''>0$ depends on $\tau$ and~$b$.  Choosing $C< \min \{C', C''\}$, we have demonstrated equation~(\ref{survival-short-intervals}) for all sufficiently large $n$, and thus completed the proof of Proposition~\ref{survival-thm}.
\end{proof}

%%%%%%%%% Section:
%%%%%%%%%  Duality
\section{Duality}
\label{section:duality}
It is well known that the contact process is self-dual in the following sense.  For any two sets of vertices $A,B\subset V$ and $t> 0$
\begin{equation}
\label{duality-eqn}
\prob{\xi_t^A \cap B \neq \emptyset}{} = \prob{\xi_t^B \cap A \neq \emptyset}{}.
\end{equation}

For our purposes, the way to understand the contact process duality is through the graphical representation of the process (see, for example, Part I of~\cite{Liggett:1999}, which we paraphrase here).  To each vertex $v\in V$ we assign a rate $1$ Poisson process with jump times $\{T^v_k\}_{k=1}^{\infty}$, and to each ordered pair of vertices $(u,v)$ joined by an edge in $G$ ($\{u,v\}\in E$) we assign a rate $\lambda$ Poisson process with jump times $\{T^{(u,v)}_k\}_{k=1}^{\infty}$.  All of these Poisson processes are independent of one another.  To construct the contact process graphically, we begin by drawing the space-time axes $G \times [0,\infty)$.  For each $k\in \N$ and $v\in V$, we draw a recovery dot, $\bullet$, at the point $(v,T_k^v) \in G\times [0,\infty)$.  For each $k\in \N$ and ordered pair $(u,v)$ such that $\{u,v\}\in E$, we draw an infection arrow, $\rightarrow$, from $(u,T_k^{(u,v)})$ to $(v, T_k^{(u,v)})$.

We say that there is an active path from $(v_0,t_1)$ to $(v_{\ell},t_2)$, with $t_1<t_2$, if there is a sequence of arrows between $v_0$ and $v_{\ell}$, 
\begin{equation*}
t_1\leq T_{k_1}^{(v_0, v_1)} < T_{k_2}^{(v_1,v_2)} < \cdots < T_{k_{\ell}}^{(v_{\ell-1}, v_{\ell})} \leq t_2,
\end{equation*}
such that there are no recovery dots encountered along the way, 
\begin{align*}
\{ T_k^{v_i} \in [T_{k_{i}}^{(v_{i-1},v_{i})}, T_{k_{i+1}}^{(v_{i},v_{i+1})}) \} = \emptyset  \ \ \forall i = 1,\ldots, \ell-1 \\
\{T_k^{v_0} \in [t_1, T_{k_1}^{(v_0,v_1)}) \} = \{T_k^{v_{\ell}} \in [T_{k_{\ell}}^{(v_{\ell-1},v_{\ell})}, t_2] \} = \emptyset.
\end{align*}
We now have that $v\in \xi_t$ if and only if there is a vertex $u\in \xi_0$ such that there is an active path from $(u,0)$ to $(v,t)$.  Therefore, by tracing the arrows in reverse, we can determine for each $v\in V$ whether $v\in \xi_t$.  Because the Poisson processes determining arrows in each direction between pairs of adjacent vertices are iid, we have that the sets
\begin{align*}
&\{u \in V \ | \ \exists\, v\in A \text{ s.t. there is an active path from } (u,0) \text{ to } (v,t)\} \ \text{ and}\\
&\{v \in V\ |\ \exists\, u\in A \text{ s.t. there is an active path from } (u,0) \text{ to } (v,t)\}
\end{align*}
are equal in distribution, which is equivalent to the self duality of the contact process as stated in equation~(\ref{duality-eqn}).  The dual process of the contact process can thus be viewed as a reversal of the arrows in the graphical representation.

Let $\{u_b, v_b\} \in E$ be the bridge edge between $V_1$ and $V_2$, such that $u_b \in V_1$ and $v_b \in V_2$.  We will use the graphical representation of the dual of the contact process to determine whether the vertex $u_b$ is infected at the times $\tbridge_k$.  The idea is that by Theorem~\ref{survival-thm} there are at least $\eps n$ vertices infected in $V_1$ at the times $\tbridge_k- r - \tmix$ (where $\tmix = O(\log \log n)$ will be defined later).  By Lemma~\ref{early-cp-lem}, the probability that the dual process started at $(u_b,\tbridge_k)$ survives to time $\tbridge_k - r$ is approximately $1- 1/b$.  If the dual process does survive, then by coupling the dual process with a random walk process, we will show that in time $\tmix$ (going backwards in time still) at least one of the active vertices in the dual process will coincide with one of the $\eps n$ infected vertices in the contact process at time $\tbridge_k - r - \tmix$ with high probability.  Therefore, $u_b$ will be infected at time $\tbridge_k$, which will immediately result in $v_b$ becoming infected.  In turn, $v_b$ starts a widespread infection in $V_2$ with probability close to $1 - 1/b$ by Lemma~\ref{early-cp-lem} and Theorem~\ref{survival-thm}.

To begin, we let $\zeta_t^k$ be the dual process on $G_1$ started at $(u_b,\tbridge_k)$.  Note that $\{\tbridge_k\}_k$ are random times, but they are independent of $\{T_k^{(u,v)}\}_k$ and $\{T_k^v\}_k$ for all ordered edges $(u,v)$ and vertices $v$ in $G_1$. The interpretation of $\zeta_t^k$ is that $u_b \in \xi_{\tbridge_k}$ if and only if $\zeta_t^k \cap \xi_{\tbridge_k - t} \neq \emptyset$.  Also note that $\prob{\zeta_t^k \cap B \neq \emptyset}{u_b} = \prob{\xi_t \cap B \neq \emptyset}{u_b}$ for all $B\subset V_1$ by the duality equation~(\ref{duality-eqn}).  Therefore we can apply Lemma~\ref{early-cp-lem} to $\zeta_t^k$ to say that $\abs{\zeta_r^k}>0$ with probability differing from $(1-1/b)$ by at most $O((\log n)^{-1})$, and when $\abs{\zeta_r^k}>0$,  $\abs{\zeta_r^k} > \eta_2 \log n$ with probability $1-O((\log n)^{-1})$.  The purpose of the next lemma is to show, via a coupling with random walks, that whenever $\abs{\zeta_r^k} > \eta_2 \log n$, we have $u_b \in \xi_{\tbridge_k}$ with high probability.

\begin{lem}
\label{dual-lem}
If $\zeta_{t}^k$ is the dual of the contact process started at $(u_b,\tbridge_k)$, $\eps>0$ and $\tmix = \frac{2}{b+1} \log \log n$ then for all $\delta >0$
\begin{align*}
\gprob{\sup_{A\subset V_1,\, \abs{A}\geq \eps n}\probGone{\zeta_{r+\tmix}^k \cap A = \emptyset \ \middle| \ \abs{\zeta_r^k} \geq \eta_2 \log n}{} > \delta}\to 0.
\end{align*}
\end{lem}

For the proof of Lemma~\ref{dual-lem} we will first need to construct the coupling between $\zeta_t^k$ and a random walk process, then prove some facts about this process.  Let $\X_t = \{X_t^1, X_t^2, \ldots, X_t^{\eta_2 \log n} \}$ be the locations of $\eta_2 \log n$ independent, continuous time, simple random walks on $G_1$ that independently die at rate $1$.  That is, for each $i$, $X_t^i$ is the continuous time random walk on $G_1$ that holds at a vertex $u\in V_1$ for time Exp($\lambda \deg(u)$), then jumps to $v\sim u$ with probability $1/\deg(u)$, and which has a total life span distributed as Exp(1).  When a walker dies, we remove it from the set $\X_t$.

We can couple $\X_t$ with $\zeta_{r+t}^k$ until $t = T_{\text{die}} \wedge T_{\text{collide}}$, where $T_{\text{die}} = \inf \{t : \X_t = \emptyset \}$ (all of the random walkers die out), and $T_{\text{collide}} = \inf \{t : X_t^i = X_t^j \text{ for some } i\neq j\}$ (two random walks collide).  First, let $X_0^i$ be the $i^{\text{th}}$ element of the lexicographical ordering of the vertices in $\zeta_r^k$.  We then update $X_t^i$ according to the infection arrows by always following the first arrow to arrive in the correct orientation (going backwards in time, so this is actually the most recent arrow infecting the current vertex at which the walker resides).  That is, if $X_t^i$ is at the vertex $u$ at time $t$, then $X_t^i$ will jump to $v$ at (random walk) time $T_k^{(u_b,v_b)} - r -T_\ell^{(v,u)}$ if the arrow from $v$ to $u$ is the next encountered:
\begin{equation*}
T_\ell^{(v,u)} = \max_{w\sim u} \max_j \{ T_j^{(w,u)} \ | \  T_j^{(w,u)}< T_k^{(u_b,v_b)} - r - t \}.
\end{equation*}
If, while at the vertex $u$, $X_t^i$ encounters a recovery dot before an arrow into $u$ (there is a $j$ so that $\tbridge_k -r-t >T_j^u>T_\ell^{(v,u)}$) then the random walk dies.  It is clear that this constructs the random walk $X_t^i$ as described above, since the waiting time until the first infection arrow into the vertex $u$ is distributed as the minimum of $\deg(u)$ random variables with Exp($\lambda$) distribution, which is Exp($\lambda \deg(u)$), and the recovery dots appear at rate $1$.

Under this coupling $\X_t \subset \zeta_{r+t}^k$.  When a collision occurs between two random walkers, however, our coupling would cause those walks to stick together for all time.  To avoid this, at the time of the first collision, $T_{\text{collide}}$, we stop the coupling between $\X_t$ and $\zeta_{r+t}^k$, and instead let each of the random walks proceed independently (of the other random walks and of $\zeta_{r+t}^k$).  This way, at time $\tmix$ the locations of the surviving random walks are independent, and as long as $\tmix < T_{\text{collide}}$ they are still coupled with $\zeta_{r+t}^k$.

The proof of Lemma~\ref{dual-lem} has three basic ingredients.  First, many of the random walkers will survive for time $\tmix$.  Second, most of the random walkers will make at least $\kmix$ jumps by time $\tmix$, so by Lemma~\ref{mixing-lem} their locations will be almost uniformly distributed, so intersection with $A$ is imminent.  Finally, a collision is unlikely to occur before time $\tmix$, so the random walk process will still be coupled with $\zeta_{r+t}^k$.

\begin{proof}
We assume that $\abs{\zeta_r^k}\geq \eta_2 \log n$, $A\subset V_1$ with $\abs{A}\geq \eps n$ is fixed, and we have the coupling described above between $\X_t$ and $\zeta_{r+t}^k$.

First we observe that the random walkers are mutually independent, and so are their death clocks.  Since each random walker dies at rate $1$, the number of random walks that survive to time $\tmix = \frac{2}{b+1} \log \log n$ is Binomial($\eta_2 \log n, (\log n)^{-2/(b+1)}$).  So by Chebychev's inequality, the number of random walks that survive until time $\tmix$ is at least $\frac{1}{2}\eta_2 (\log n)^{(b-1)/(b+1)}$ with probability $1 - O((\log n)^{-(b-1)/(b+1)})$.

Next, we want all of the random walkers to make at least $\kmix$ steps so we can apply Lemma~\ref{mixing-lem}.  The probability that the $i^{\text{th}}$ walker, $X_t^i$, jumps fewer than $\kmix$ times by time $\tmix$ (ignoring whether the walker survives to time $\tmix$, as these events are independent) is at most
\begin{align*}
\sum_{k< \kmix} \frac{1}{k!} e^{-b(1-n^{-a/3})\tmix} [b(1-n^{-a/3})\tmix]^k &\leq e^{1-b(1-n^{-a/3})\tmix} [b(1-n^{-a/3})\tmix]^\kmix \\
& = O\left((\log n)^{-2b/(b+1)}(\log \log n)^\kmix \right)\\
\end{align*}
since the minimum degree of $G_1$ is at least $n^a - n^{2a/3}$ with high probability.  Therefore, all of the surviving random walks will make at least $\kmix$ jumps by time $\tmix$ with probability $1 - O((\log n)^{-(b-1)/(b+1)} (\log\log n)^{\kmix})$. 

Conditional on the events that $X_t^i$ survives to time $\tmix$ and makes at least $\kmix$ jumps, then the probability that $X_{\tmix}^i \in A$ is at least
\begin{align*}
\frac{\abs{A}}{n} (1 - 4n^{-a/3}) \geq \eps(1 - 4n^{-a/3})
\end{align*}
for asymptotically almost every $G_1$ by Lemma~\ref{mixing-lem}.  Let $E$ be the event that there are at least $\frac{1}{2}\eta_2 (\log n)^{(b-1)/(b+1)}$ random walkers that survive to time $\tmix$ and make at least $\kmix$ jumps. Then the probability that none of the random walkers hits the set $A$ is
\begin{align}
\nonumber \prob{ \vect{X}_{\tmix} \cap A = \emptyset}{} &\leq \prob{ \vect{X}_{\tmix} \cap A = \emptyset | E}{} + \prob{E^c}{} \\
\nonumber &\leq (1-\eps (1 - 4n^{-a/3}))^{\frac{1}{2}\eta_2 (\log n)^{(b-1)/(b+1)}} + \prob{E^c}{} \\
\label{rw-hitting-prob}
& = O((\log n)^{-(b-1)/(b+1)} (\log\log n)^{\kmix}).
\end{align}

Now we only need to check that the coupling between $\vect{X}_t$ and $\zeta_{r+t}^k$ has not been violated before time $\tmix$.  That is, we need to check that $T_{\text{collision}} > \tmix$. Observe that $T_{\text{collision}} \succ T'$ where $T' \sim$~Exp($\lambda(\eta_2 \log n)^2$), since at any time there are fewer than $\abs{\vect{X}_t}^2 \leq (\eta_2 \log n)^2$ directed edges connecting vertices occupied by the random walkers, and each directed edge has a rate $\lambda$ Poisson clock.  Therefore, the probability that a collision occurs by time $\tmix$ is
\begin{align}
\nonumber \prob{T_{\text{collision}} \leq \tmix}{} &\leq \prob{T' \leq \tmix}{}\\
\nonumber & = 1 - \exp\left[-\tmix \lambda (\eta_2 \log n)^2 \right] \\
\nonumber & = 1 - \exp \left[- \frac{2b\eta_2^2}{b+1} n^{-a} (\log n)^2 \log\log n \right] \\
\label{collisions}
&= O( n^{-a} (\log n)^2 \log\log n).
\end{align}

Piecing together equations~(\ref{rw-hitting-prob}) and~(\ref{collisions}) shows that the coupling between $\zeta_{r+t}^k$ and $\vect{X}_t$ will not be violated before time $\tmix$, so $\zeta_{r+\tmix}^k \cap A \neq \emptyset$ with probability
$$1-O((\log n)^{-(b-1)/(b+1)} (\log\log n)^{\kmix})$$
for asymptotically almost every $G_1$ whenever $\abs{\zeta_{r}^k}\geq \eta_2 \log n$.
\end{proof}

\section{Proof of Theorem~\ref{main-thm}}
\label{main-proof}
\begin{proof}
First, we observe that by Lemma~\ref{early-cp-lem}, $\prob{\abs{\xi_r} > 0}{v_0} = 1-\frac{1}{b} - o(1)$.  Then by Proposition~\ref{survival-thm}, if $\abs{\xi_r} > 0$ then $\abs{\xi_t} > \eps n$ for all times $t\in [\eta_3 \log n, e^{cn}]$ with high probability.  Note that the number of vertices that are used up by the dual processes, $\{\zeta_t^k\}_k$, at any time is at most $O((\log n)^4)$ by Lemma~\ref{early-cp-lem} (when we observe the dual process, it is only allowed to grow for time $r = O(\log \log n)$).  This is not a problem, since we have accounted for the fact that $\xi_t$ must avoid these vertices in equation~(\ref{eps-def}).

Now we consider the sequence of times at which the bridge edge may transmit the infection from $u_b$ to $v_b$, $\{T_{k}^{(u_b, v_b)}\}_k$.  The interarrival times, $T_{1}^{(u_b, v_b)}$ and $T_{k}^{(u_b, v_b)} - T_{k-1}^{(u_b, v_b)}$ for $k\geq 2$, are independent with distribution Exp($\lambda$).  If any of the first $\log n$ interarrival times are smaller than $2r+\tmix$, then we will not have enough time to allow the dual process started at $T_{k}^{(u_b, v_b)}$ to be well mixed and guarantee that $\xi_{T_{k}^{(u_b, v_b)}}(u_b)$ is essentially independent for each $k$.  The reason for the $2r$ term is that once the infection reaches the second component, we need additional time $r$ to see whether it survives in that component.  Fortunately the probability that any of the first $\log n$ interarrival times is smaller than $2r+\tmix = O(\log \log n)$ is at most $O(n^{-a} \log n \log \log n)$.

By Lemmas~\ref{early-cp-lem} and~\ref{dual-lem}, the dual process $\zeta_t^k$, started at space-time $(u_b, T_{k}^{(u_b, v_b)})$, will intersect with $\xi_{T_{k}^{(u_b, v_b)}-r-\tmix}$ with probability $1 - \frac{1}{b} - o(1)$.  If successful, the result is that $u_b$ is infected at time $T_{k}^{(u_b, v_b)}$, and therefore $v_b \in V_2$ becomes infected at this time.  By applying Lemma~\ref{early-cp-lem} and Theorem~\ref{survival-thm} to the infection in $G_2$ started at $v_b$, this will lead to a wide-spread infection of $G_2$ (with greater than $\eps n$ infected vertices in $G_2$) by time $T_{k}^{(u_b, v_b)} + \eta_3 \log n$ with probability $1-\frac{1}{b} - o(1)$.  The first $\log n$ interarrival times of $\{T_{k}^{(u_b, v_b)}\}_k$ all exceed $2r+\tmix = O(\log\log n)$ with high probability, and conditional on this, the probability that the infection has not yet spread to the second component following the $k^{\text{th}}$ time $T_{k}^{(u_b, v_b)}$ ($k\leq \log n$) is $\left[1-\left(1-\frac{1}{b} - o(1)\right)^2\right]^{k}$.  The probability that $N\sim$~Geometric$\left(\left(1-\frac{1}{b} - o(1)\right)^{2}\right)$ exceeds $\log n$ is $O(n^{- (1-1/b-o(1))^2})$, so we need not worry about the interarrival times exceeding $2r+\tmix$.

Therefore the distribution of the time
\begin{align*}
\tau := \inf \{t >0 : \ \abs{\xi_t^{v_0} \cap V_2} > \epsilon n \},
\end{align*}
conditional on $\abs{\xi_r}>0$, is a sum of a Geometric$\left(\left(1-\frac{1}{b} - o(1)\right)^{2}\right)$ number of independent Exp($\lambda$) random variables, which is Exp$\left(\left(1-\frac{1}{b} - o(1)\right)^{2} \lambda \right)$, plus some time $O(\log n)$ to account for waiting for the process to grow to size $\eps n$ in each subgraph.  This completes the proof for a single bridge edge, since $\tau$ is infinite when $\abs{\xi_r} = 0$, as $\prob{T_1^{(u_b, v_b)} < \eta_3 \log n}{} = O(n^{-a}\log n)$.

The proof for the multiple-bridge case, $\abs{B} = o(n^a/\log n \log\log n)$, is nearly identical if we replace $T_k^{(u_b,v_b)}$ with $T_k^B := \min \{T_k^{(u_b,v_b)} : (u_b,v_b)\in B, u_b\in V_1, v_b\in V_2\}$, the sequence of times at which {\em some} bridge edge attempts to transmit the infection.  All that needs to be checked is that the first $\log n$ interarrival times of consecutive transmission attempts are larger than $2r+\tmix = O(\log\log n)$.  This is straightforward to verify, and the details are left to the reader.
\end{proof}

	\bibliographystyle{siam}
	\bibliography{proposal_cite}

\appendix
\section{Branching Processes}

In this section we prove Lemmas~\ref{bp-size-lem} and~\ref{bp-max-size} regarding the size of a supercritical, continuous-time branching process, $Z_t$, in which each individual gives birth to a single offspring at rate $\beta>1$ and dies at rate 1.  Let $\pr{\cdot}{}$ be the probability measure associated with this process, and for all $i,j \in \Z_{\geq 0}$ and $t\geq 0$, let $P_{ij}(t) := \pr{Z_t = j \ \middle | \ Z_{0} = i}{} =: \pr{Z_t=j}{i}$.  %Since $P_{ij}(\tau, t) \deq P_{ij}(0, t-\tau)$, let $P_{ij}(t) := P_{ij}(0, t)$.  Also, let $F(s, \tau, t) = \sum_k P_{1k}(\tau, t) s^k$ and $F(s, t) = \sum_k P_{1k}(t) s^k$ be the respective probability generating functions.

%%%%% BP Survival Lemma
\begin{lem}
\label{bp-survival-lem}
Let $Z_t$ be a branching process where each individual gives birth to a single offspring at rate $\beta>1$ and dies at rate $1$.  If $Z_0 = \ceil{3 \log_{\beta} n }=: \alpha$ then
\begin{align*}
\pr{Z_t = 0 \text{ eventually}}{\alpha} \leq n^{-3}.
\end{align*}
\end{lem}

\begin{proof}
If $Z_0 = 1$, then $\pr{Z_t = 0 \text{ eventually}}{1} = 1/\beta$ \cite{harris-bp}.  If $Z_0 = \alpha$ then all of the families are mutually independent, so

\begin{equation*}
\pr{Z_t = 0 \text{ eventually}}{\alpha} = \beta^{-\alpha} \leq n^{-3}.
\end{equation*}
\end{proof}

%%%%% BP Size bounds lemma
\begin{lem}
\label{bp-size-lem}
Let $Z_t$ be a branching process where each individual gives birth to a single offspring at rate $\beta>1$ and dies at rate $1$.  When $Z_0 = 1$, if $r = \frac{2}{\beta-1} \log \log n$ and $s = \frac{3}{\beta - 1} \log n$, then for any constants $C,c>0$
\begin{align}
\pr{0 < Z_{r} \leq C \log n}{1} &= O((\log n)^{-1}) 
\label{bp-size-lb} \\
\pr{Z_r > \frac{2\beta}{\beta-1} (\log n)^3}{1} &= O(n^{-2}).
\label{bp-size-ub} \\
\pr{0 < Z_s \leq cn}{1} &= O(n^{-2})
\label{bp-size-biglb}
\end{align}
\end{lem}

The proof uses the following bounds.

%%%%% Log bounds lemma
\begin{lem}
\label{log-lem}
If $0<y<\frac{1}{2}$ then
\begin{equation}
\label{log-bound}
-y - y^2 \leq \ \log \left( 1 - y \right)\ \leq -y - \frac{y^2}{2}.
\end{equation}
\end{lem}
\begin{proof}
By using the series expansion:
\begin{equation*}
\log (1 - y) = - \sum_{k = 1}^{\infty} \frac{y^k}{k},
\end{equation*}
the upper bound in (\ref{log-bound}) is immediate by truncating after the second term in the series.  Using that $y \leq 1/2$, the lower bound in (\ref{log-bound}) follows from:
\begin{align*}
\log(1-y) &\geq - y - \frac{y^2}{2} \left( 1 + y + y^2 + \cdots \right) \\
&\geq -y - \frac{y^2}{2} (2).
\end{align*}
\end{proof}

%%% Proof of BP size lemma.
We now prove Lemma~\ref{bp-size-lem}.
\begin{proof}
The transition probabilities for $Z_t$ can be computed exactly, as in Chapter V of~\cite{harris-bp}:
\begin{align*}
P_{10}(t) &= \frac{1 - e^{-(\beta - 1)t}}{\beta - e^{-(\beta - 1)t}} \\
P_{1k}(t) &= [1 - P_{10}(t)]\ [1 - \eta(t)]\ \eta(t)^{k-1} \\
\eta(t) &= \frac{1 - e^{-(\beta - 1)t}}{1 - \frac{1}{\beta}e^{-(\beta - 1)t}}.
\end{align*}
So we have at time $r = \frac{2}{\beta-1} \log \log n$ that
\begin{align}
\nonumber \pr{0 < Z_r \leq C \log n}{1} &= \sum_{k=1}^{C \log n} P_{1k}(r) \\
&= [1 - P_{10}(r)]\cdot \left[1 - \eta(r)^{C \log n} \right].
\label{bp-size-lb1}
\end{align}
We apply Lemma~\ref{log-lem} to obtain:
\begin{align*}
\log\left[ 1 - e^{-(\beta - 1)r} \right] = \log\left[1 - (\log n)^{-2} \right] & \geq - (\log n)^{-2} - (\log n)^{-4} \\
\log\left[ 1 - \frac{1}{\beta}e^{-(\beta - 1)r} \right] = \log\left[1 - \frac{1}{\beta} (\log n)^{-2} \right] &\leq -\frac{1}{\beta} (\log n)^{-2} - \frac{1}{2\beta^2} (\log n)^{-4}
\end{align*}
whenever $n \geq 5$.  Subtracting the second line from the first, then multiplying both sides by $C \log n$ gives
\begin{equation*}
\log \eta(r)^{C \log n} \geq -C\left(1 - \frac{1}{\beta}\right) (\log n)^{-1} - O\left( (\log n)^{-3} \right).
\end{equation*}
Combining equation~(\ref{bp-size-lb1}) with this bound and the fact that $1 - e^{-x} \leq x$ for $x \geq 0$ (this follows from the upper bound of Lemma~\ref{log-lem}) proves equation~(\ref{bp-size-lb}):
\begin{align*}
\pr{0 < Z_r \leq C \log n}{1} &\leq 1 - \exp\left[ \log \eta(r)^{C \log n}\right] \\
& \leq C\left(1 - \frac{1}{\beta}\right) (\log n)^{-1} + O\left( (\log n)^{-3} \right).
\end{align*}
Equation~(\ref{bp-size-biglb}) is proved in the same way.  To prove equation~(\ref{bp-size-ub}) we begin with
\begin{align*}
\pr{Z_r > \frac{2\beta}{\beta-1} (\log n)^3}{1} &= \sum_{k>\frac{2\beta}{\beta-1} (\log n)^3} P_{1k}(r) \\
&= [1 - P_{10}(r)]\cdot \eta(r)^{\frac{2\beta}{\beta-1} (\log n)^3}.
\end{align*}
Now applying Lemma~\ref{log-lem} in the same way as above (but with the upper and lower bounds reversed) yields:
\begin{align*}
\pr{Z_r > \frac{2\beta}{\beta-1} (\log n)^3}{1} &\leq \exp \left[ \log \eta(r)^{\frac{2\beta}{\beta-1} (\log n)^3} \right] \\
&\leq \exp \left[ - \frac{2\beta}{\beta-1} \left(1 - \frac{1}{\beta}\right) \log n + O\left( (\log n)^{-1} \right) \right] \\
& = O(n^{-2}).
\end{align*}
\end{proof}

% BP maximum size lemma.
\begin{lem}
\label{bp-max-size}
Using the same setup as in Lemma~\ref{bp-size-lem}, let $\eta_1 = \frac{6\beta}{(\beta-1)\log \beta}$, then
\begin{equation*}
\pr{\sup_{0\leq t\leq r} Z_t \geq \eta_1 (\log n)^4}{1} = O\left(n^{-2} \right).
\end{equation*}
\end{lem}

\begin{proof}
For the duration of the proof of Lemma~\ref{bp-max-size} let $T := \inf \{t \ | \ Z_t \geq (\log n)^4 \}$.  Then the claim in Lemma~\ref{bp-max-size} is equivalent to $\prob{T \leq r}{1} = O\left(n^{-2} \right)$.  Our strategy is to use equation~(\ref{bp-size-ub}) in Lemma~\ref{bp-size-lem} to say that $Z_r$ can be at most $O((\log n)^3)$, then use the Strong Markov Property to say that if $Z_t$ exceeds $\eta_1 (\log n)^4$ at any time $t\leq r$ then it is unlikely to drop below size $O((\log n)^3)$ by time $r$.
\begin{align*}
\pr{T \leq r}{1} &\leq \pr{T\leq r,\ Z_r \leq \frac{2\beta}{\beta-1} (\log n)^3}{1} + \pr{Z_r > \frac{2\beta}{\beta-1} (\log n)^3}{1} \\
&\leq \pr{Z_r \leq \frac{2\beta}{\beta-1} (\log n)^3 \ \middle| \ T\leq r}{1} + O(n^{-2}) \\
&\leq \pr{\inf_{t\leq r} Z_t \leq \frac{2\beta}{\beta-1} (\log n)^3}{\eta_1(\log n)^4} + O(n^{-2}) \\
&\leq 1 - [1 - \pr{Z_r = 0}{3\log_{\beta} n}]^{\frac{2\beta}{\beta-1} (\log n)^3} + O(n^{-2}) \\
&\leq 1 - \left[1 - \frac{1}{n^3}\right]^{\frac{2\beta}{\beta-1} (\log n)^3} + O(n^{-2}) = O\left( n^{-2} \right).
\end{align*}
We applied the Strong Markov Property and translation invariance of the branching process at the third line above.  The fourth line follows by observing that for the branching process to transition from size $\eta_1(\log n)^4$ to size $\frac{2\beta}{\beta-1} (\log n)^3$, then at least one of the $\frac{2\beta}{\beta-1} (\log n)^3$ sets of $3 \log_{\beta} n$ individuals must go extinct.  The fifth line follows from Lemma~\ref{bp-survival-lem}.
\end{proof}

\end{document}